\documentclass[reqno, 11pt]{amsart}

\makeatletter
\@namedef{subjclassname@2020}{%
  \textup{2020} Mathematics Subject Classification}
\makeatother

\usepackage{ulem}

\usepackage{amsfonts, amsthm, amsmath, amssymb}
\usepackage{hyperref}
\hypersetup{colorlinks=false}

\usepackage{graphicx}

\usepackage{bbm}  

 \usepackage{tabu}

\usepackage[margin=1in]{geometry}

\usepackage{helvet}

\RequirePackage{mathrsfs} \let\mathcal\mathscr

\usepackage{hyperref}
\usepackage{dsfont}
\usepackage{upgreek}
\usepackage{mathabx,yfonts}
\usepackage{enumerate}
\usepackage{xcolor}

\numberwithin{equation}{section}

\newtheorem{theorem}{Theorem}[section] 
\newtheorem{lemma}[theorem]{Lemma}

\theoremstyle{definition}

 \newtheorem*{acknowledgements}{Acknowledgements}
\newtheorem{remark}[theorem]{Remark}

\renewcommand{\emph}[1]{\textit{#1}}

\renewcommand{\phi}{\varphi}

\renewcommand{\leq}{\leqslant}

\renewcommand{\geq}{\geqslant}

\renewcommand{\c}{\mathbf{c}}

 \renewcommand{\b}{\mathbf{b}}

\DeclareSymbolFont{bbold}{U}{bbold}{m}{n}
\DeclareSymbolFontAlphabet{\mathbbold}{bbold}

\renewcommand{\P}{\mathbb{P}}

\newcommand{\Z}{\mathbb{Z}}

\renewcommand{\b}{\mathbf}
\renewcommand{\c}{\mathcal}
\renewcommand{\epsilon}{\varepsilon}

\renewcommand{\leq}{\leqslant}
\renewcommand{\geq}{\geqslant}
\renewcommand{\#}{\sharp}

\newcommand{\beq}[2]
{
\begin{equation}
\label{#1}
{#2}
\end{equation}
}

\title
[Sums of coefficients of $\operatorname{GL}(m)$-automorphic forms]
{Square-root cancellation for 
sums of coefficients of $\operatorname{GL}(m)$-automorphic forms over values of random polynomials}

\author{Dimitrios Chatzakos} 
\address{Department of Mathematics\\ 
University of Patras
\\ 26 504, Patras, Greece}
\email{dchatzakos@math.upatras.gr}

\author{Efthymios Sofos} 
\address{Dipartimento di Matematica\\
Universit{\`a} di Roma
\\Tor Vergata\\
 Via del  la Ricerca Scientifica
 \\ 00133, Rome, Italy}
\email{efthymios.sofos@uniroma2.eu}

\date{}

\begin{document} 
\begin{abstract}
For arbitrary $m\geq 2$ and $d\geq 1$, we establish  square-root cancellation for sums of Fourier coefficients of $\mathrm{GL}(m)$-automorphic forms averaged
over values of random degree-$d$ polynomials.
\end{abstract}

\maketitle

\setcounter{tocdepth}{1}
\tableofcontents

\section{Introduction}   \label{s:intro}  
The study of sums of arithmetic functions evaluated along polynomial sequences, given by 
\begin{equation}\label{mainsums}
\sum_{\substack{  n \leq x\\P(n)>0}} g( P(n) ),\end{equation}
where $g:\mathbb N
\to \mathbb C$ and $P \in \mathbb Z[t]$, 
represents a cornerstone of analytic number theory with  
significant connections to other fields. A very important case is when $g$ is the normalized Fourier coefficient  $\lambda_\pi$ of an automorphic form $\pi$ (which are closely related to Hecke eigenvalues). For linear polynomials $P$, the study of sums of the form \eqref{mainsums} reduces to mean value estimates for those Fourier. This problem has been attacked using the theory of summation formulas and $L$-functions (see for instance the survey \cite{MillerSchmid} for the rich history of this problem).

When $\pi$ is a $\mathrm{GL}(2)$ cusp form and $P$ is quadratic, sums of the form \eqref{mainsums}
were studied by Blomer \cite{blom},
Templier \cite{nite}
and Templier--Tsimerman \cite{ttt}.
The work in \cite{blom}
obtained the error term $O(x^{6/7+\epsilon})$
via Kuznetsov’s trace formula, while 
in \cite{ttt} the bound 
$O(x^{1/2+\theta+\epsilon})$  is proved 
via sup norm bounds for Fourier coefficients;
here $\theta$ is the exponent depending only on the specific progress on the  
Selberg's eigenvalue conjecture. Assuming this conjecture one can take $\theta=0$,
while $\theta\leq 7/64$ is currently known   
\cite{blomerbrumley,kim}. We also refer to \cite{sarnakletter} for an extended discussion on sums arising from higher degree polynomials. 
  
As the general problem is   intractable, it is natural to examine its behaviour under generic conditions. 
In the present work, we introduce   alternative methods 
to establish unconditionally
that for any $\mathrm{GL}(m)$ cusp form and for almost all polynomials $P$ of arbitrary degree, we have square-root cancellation, i.e. 
\begin{eqnarray} \label{mainresultheuristic}
\sum_{\substack{ n \leq x\\P(n)>0}} \lambda_\pi( P(n) ) 
=O(x^{1/2}).
\end{eqnarray}
In order to do so, we order the polynomials $P\in \mathbb Z[t]$
by the height
\begin{eqnarray*}
|P|:=\textrm{ maximum of absolute values of its 
coefficients}. 
\end{eqnarray*}

In order to state our results, fix arbitrary $d \in \mathbb N, m \in \mathbb N_{\geq 2}, \epsilon>0$ and denote $$\alpha: = 
\min \left\{ \frac{2}{2d(2m+3) + 3m + 4}, \frac{2}{d(m^2 - 1)} \right\}.$$  
When $m\geq 7$ and $d\neq 1$
we always have $\alpha= 2/(d(m^2 - 1))$. Our first main result is the following. 
\begin{theorem}\label{thm:general_m} 
Let $\pi $ be an 
automorphic irreducible cuspidal 
representation of $\mathrm{GL}(m) $ 
over $\mathbb Q$ with unitary central character. 
Then for all $ x \leq H^{\alpha-\epsilon}$
  we have 
$$\frac{1}{\#\{P\in \Z[t]: \deg(P)=d,
|P|\leq H\}}
\sum_{\substack{ P\in \Z[t], \\ \deg(P)=d, |P|\leq H  }}
\left| \sum_{\substack{ n \leq x\\P(n)>0} } \lambda_\pi(P(n))\right|^2 
\ll
x$$ where the implied constant depends
only on $d,\pi$ and $\epsilon$.
\end{theorem}   
This is proved in \S\ref{mentre dormi}
by feeding   work
of Jiang--L\"u--Wang \cite{Jiang} 
into Theorem \ref{gln} (verified in \S\ref{s:senzaarcigrandi}.) The condition $P(n) > 0$ is 
imposed because otherwise $\lambda_\pi(P(n))$ would not be defined for all $\mathrm{GL}(m) $ forms; 
however, the proof is 
easily adapted to accommodate these cases whenever the 
terms are well-defined.

Let us denote by $C_{\pi}$ the limit
$$ C_\pi= \lim_{x\to\infty} \frac{1}{x} 
\sum_{n\leq x }|\lambda_\pi(n)|^2.$$
It is known from Rankin--Selberg theory that this constant essentially equals 
the residue of the 
Rankin--Selberg $L$-function $L(s, \pi \times \tilde{\pi})$ at $s=1$,
where  $\tilde{\pi}$ is the 
dual of $\pi$. In the $\mathrm{GL}(2)$ and  $\mathrm{GL}(3)$ 
cases we show that for random polyomials $P$
one has  an asymptotic rather than a bound:
$$\left|\sum_{\substack{ n \leq x\\P(n)>0} }  \lambda_\pi( P(n) ) \right|
\sim (C_\pi x)^{1/2}.$$ We  
improve   cases of Theorem \ref{thm:general_m} 
in the following manner: 
Let  
$$ \beta(2)=\frac{ 1}{2+3d}, \ \  \
\beta(3):=\frac{1}{7d+4} $$
and fix $d\in \mathbb N$,
$\epsilon>0$ and $N>0$.  
\begin{theorem}\label{thm:small_m} 
Let $m$ be $2$ or $3$ and suppose $\pi $ is an 
automorphic irreducible cuspidal 
representation of $\mathrm{GL}(m) $ 
over $\mathbb Q$ with unitary central character. 
Then for all 
  $x \leq H^{\beta(m)-\epsilon}$
  we have 
$$\frac{1}{\#\{P\in \Z[t]: \deg(P)=d,|P|\leq H\}}
\sum_{\substack{ P\in \Z[t] \\ \deg(P)=d , |P|\leq H  }}
\left| \sum_{\substack{ n \leq x\\ P(n)>0} } g(P(n))\right|^2 
=C_\pi  x+O\left(\frac{x}{(\log x)^N}\right)$$
where the implied constant depends only on $d,\pi,\epsilon$   and $N$.  
\end{theorem} 

\begin{remark}
We would like to highlight the impressive
recent progress made on special $GL(2)$ cases of the Rankin--Selberg problem by Huang \cite{Huang} and Dasgupta--Leung--Young \cite{young}. While these important results bring us closer, they don't quite yield an immediate improvement for $\beta(2)$ in Theorem \ref{thm:small_m}, for reasons we outline in Remark \ref{eightvariables}. 
\end{remark}

The proof of Theorem \ref{thm:small_m} 
is in \S\ref{mentre dormi}. 
The $\mathrm{GL}(3) $ case 
uses
work of Miller \cite{Miller}, 
who studied the additive twists   
$$ \sum_{n\leq x } \lambda_\pi (n)  
\mathrm e^{ 2  i \gamma n}  , \ \ \ \gamma \in \mathbb R,$$
whereas for the $\mathrm{GL}(2) $  case we use 
Wilton's analogous bound (see \cite[Theorem 8.1]{MR1942691}).
This input is then fed into 
Theorem \ref{gl3} that is proved in \S \ref{s:proof1}.
It is worth noting that Theorem \ref{gl3} 
is the first result in the literature giving 
asymptotics
for the second moment of sums of \textit{general}
arithmetic functions
over values of random polynomials.

\begin{remark}
It should be noted that, in a diffenent direction, sums of the form
\begin{equation}\label{differentsums}
\sum_{\substack{ n \leq x\\P(n)>0}} 
|\lambda_{\pi}( |P(n)| )|
\end{equation}
behave quite differently as one can only hope for a logarithmic saving.
They are related to various other important questions in the theory of automorphic forms,
the QUE conjecture and Manin's conjecture on counting rational points. We refer to \cite{woo} and references therein for a detailed discussion on the case $m=2$, $d \geq 1$ and its applications. 
\end{remark}
 
\begin{remark}
The study of sums of the form \eqref{mainsums} over the values 
of random polynomials has seen recent developments. 
Work has so far focused on  
the von Mangoldt function
\cite{skoro,joni,woo1,borto}, 
the M\"obius and Liouville functions 
 \cite{joni,woo1,wilson},
 the norm representation function \cite{joni},
and the square-free indicator function \cite{shpar,broshpar,sofo}.
Square-root cancellation is not a priori obvious. Indeed, one has \textit{worse} than square-root cancellation for $ \Lambda$ \cite{borto}, \textit{better} than square-root cancellation for $  \mu^2$ \cite{sofo}, and square-root cancellation for $  \mu$ \cite{wilson,woo}.
\end{remark}
\begin{acknowledgements}
This project was inspired by
conversations 
with Matteo di Scipio 
during the  $9$th
mini symposium of the Roman Number Theory Association.
We are grateful to him and to the organisers of the symposium. 
We are also indebted to Gergely Harcos, Bingrong Huang, Subhajit Jana, Yujiao Jiang and Matthew Young for bringing relevant references to our attention and for 
insightful guidance regarding the literature. D. C. was supported by the Hellenic Foundation for Research and Innovation (H.F.R.I.) under the “3rd Call for H.F.R.I. Research Projects to support Faculty Members \& Researchers” (Project Number: 25622).
\end{acknowledgements}

\section
{Estimates for Fourier coefficients of automorphic forms} \label{secondsection}

In this section we recall all the necessary estimates for Fourier coefficients of automorphic forms that we need for the proofs of our results. 

\subsection{Rankin-Selberg theory} \label{subsection2.1}

Let $\pi$ be a cuspidal automorphic representation of $\hbox{GL}(m, \mathbb{A}_{\mathbb{Q}})$. As usually, we denote by $A_{\pi}(n_1,...,n_{m-1})$ the Fourier coefficient of $\pi$ and let 
\begin{eqnarray}
\lambda_{\pi}(n) := A_{\pi}(n, 1, 1, ..., 1)
\end{eqnarray}
normalized such that 
\begin{eqnarray*}
\lambda_{\pi} (1) = 1.    
\end{eqnarray*}
Asymptotic estimates for the second moment sums of $\lambda_{\pi}(n)$ are known using the classical Rankin-Selberg theory, i.e. the analytic properties of the $L$-function $L(s, \pi \times \tilde{\pi})$. 
In particular, for $m \geq 2$ the classical convexity error term for the Rankin--Selberg problem  in $\mathrm{GL}(m)$ for general $m$ is the following. 

\begin{lemma}\label{lem:clasconvexity}
Let $m \geq 2$ and let $\pi $ be an 
automorphic irreducible cuspidal representation of $\mathrm{GL}(m) $ over $\mathbb Q$ with unitary central character. 
For  fixed $\epsilon>0$ and all $x\geq 1 $ we have  
$$
\sum_{ n \leq x } 
|\lambda_\pi(n)|^2= C_\pi x 
+ O(x^{a_m+\epsilon})
,$$ 
with $a_m=(m^2-1)/(m^2+1)$, where the implied constant depends at most on $\epsilon$ and $\pi$ and $C_{\pi}$ equals 
the residue of the 
Rankin--Selberg $L$-function $L(s, \pi \times \tilde{\pi})$ at $s=1$.
\end{lemma}

The proof of the lemma follows along the lines of \cite[page 26]{2507.20653}, for example. It uses a Perron-type argument applied to the $L$-function $L(s, \pi \times \tilde{\pi})$ of degree $m^2$. For more details on the Fourier coefficients of automorphic forms and on higher rank Rankin--Selberg $L$-functions we refer to \cite{Goldfeld}.

\subsection{Twisted linear sums}  \label{subsection2.2}

Let $\pi(z)$ be a $\hbox{GL}(2, \mathbb{R})$ cusp form with Fourier coefficients $\lambda_\pi(n)$. Wilton's result states that for any $\gamma  \in \mathbb{R}$ the following bound holds uniformly on $\gamma$:
\begin{eqnarray*}
\sum_{n \leq X} \lambda_\pi (n)   
\mathrm 
e^{2  i\gamma n} \ll_{\pi} X^{1/2+\epsilon}.
\end{eqnarray*}
This result was improved by Jutila \cite{jutila} to $O(X^{1/2})$. When $\pi$ is a $\hbox{GL}(3, \mathbb{R})$, an analogous result was deduced by \cite{Miller}, who proved that for any $\gamma \in \mathbb{R}$ we have the following bound uniformly on $\gamma$:
\begin{eqnarray*}
\sum_{n \leq X} \lambda_\pi (n)     
\mathrm 
e^{ 2  i \gamma n}  \ll_{\pi} X^{3/4+\epsilon}.
\end{eqnarray*}
For the higher rank case $m \geq 3$ only weaker results are known. When $\pi $ is an automorphic irreducible cuspidal representation of $\mathrm{GL}(m) $ 
over $\mathbb Q$ with unitary central character, the best known uniform bound for all $a \in \mathbb{R}$ in the literature is due to Jiang, L\"u and Wang \cite{Jiang} who proved the bound
\begin{eqnarray} \label{boundhigher}
\sum_{n \leq X} \lambda_{\pi} (n)   
\mathrm 
e^{2 \pi \gamma i n} \ll_{m} \frac{X}{\log X}.
\end{eqnarray}
One may wish for a uniform upper bound in \eqref{boundhigher} with a power saving $O(X^{1-\delta_m  + \epsilon})$, for some $\delta_m >0$. As explained in \cite[Theorem~7.1]{Miller}, such a bound is equivalent to an explicit {\lq Ramanujan type\rq}-uniform upper bound for a period integral. 

Nevertheless, for twisted linear sums with rational angle $\gamma  = a/q \in \mathbb{Q}$, Jiang, L\"u and Wang \cite{Jiang} proved a weaker upper bound which suffices for our purposes. 
 
\begin{lemma}[Jiang--L\"u--Wang]\label{lem:585858}
Suppose $\pi $ is as in Lemma \ref{lem:clasconvexity}.
For any $a\in \mathbb Z, q\in \mathbb N$ and $x\geq 1 $ we have 
$$\sum_{ n\leq x } \lambda_\pi(n)  
\mathrm 
e^{2\pi i \frac{a}{q}n}\ll (qx)^{\frac{m+1}{m+2}+\epsilon},$$ where the implied constant depends at most on $\epsilon$ and $\pi$.
\end{lemma}

Finally, we recall    
the following recent improvement in the error term of the Rankin--Selberg problem:
 \cite[Corollary 1.3]{young}.
\begin{lemma}[Dasgupta--Leung--Young]
\label{lem:dasgu}
Let $\pi$ be a fixed Hecke-Maass cusp form for 
$SL_2(\Z)$. 
For  fixed $\epsilon>0$ and all $x\geq 1 $ we have  
$$
\sum_{ n \leq x } 
|\lambda_\pi(n)|^2= C_\pi x + O(x^{4/7+\epsilon})
,$$ where the implied constant depends at most on $\epsilon$ and $\pi$.\end{lemma}

\section{Second moment with minor arcs} \label{s:proof1}
Assume that there are constants
$\gamma,\delta\in [0,1),C\in \mathbb C$ 
and that $g: \mathbb N\to \mathbb C$ satisfies 
\beq{prop:1}{\sup_{\lambda\in \mathbb R}
\left|\sum_{ n\leq x} 
\mathrm e^{i \lambda n} g(n)\right|\ll x^{\gamma}}
and \beq{prop:2}{\sum_{ n\leq x} |g(n)|^2=C x+O(x^{\delta}).}
Fix any $d\in \mathbb N, \gamma,\delta,\epsilon,N>0$ and denote 
$$\eta:=\min\left\{\frac{ 1-\gamma}{1+d(1+\gamma)},
\frac{1-\delta}{d\delta}\right \}.$$

\begin{theorem}\label{gl3} For $g:\Z\to\mathbb C$ satisfying \eqref{prop:1}-\eqref{prop:2}
and all   $ x \leq H^{\eta-\epsilon}$ we have 
$$\frac{1}{\#\{P\in \Z[t]: \deg(P)=d,|P|\leq H\}}
\sum_{\substack{ P\in \Z[t] \\\deg(P)=d , \\ |P|\leq H  }}
\left| \sum_{\substack{1\leq n \leq x\\P(n)>0} } g(P(n))\right|^2 
= C x+O\left(\frac{x}{(\log x)^N}\right).$$
The implied constant depends at most on $N,C,\gamma,\delta,\epsilon$ and 
the implied constants in
\eqref{prop:1}-\eqref{prop:2}.\end{theorem} 
The proof is based on the ideas in  \cite[\S 3]{skoro}. We start by 
expanding 
$$\sum_{\substack{ P\in \Z[t],
\deg(P)=d \\ |P|\leq H  }}
\left| \sum_{\substack{ n \leq x\\P(n)>0} } g(P(n))\right|^2
=\sum_{ n_1,n_2\leq x}
\sum_{\substack{ P\in \Z[t], P(n_1),P(n_2)>0 \\\deg(P)=d , |P|\leq H  }}
\overline{g(P(n_1))} g(P(n_2)) $$ 
and then 
 use the circle method identity 
to write the sum over $P$ in 
the right-hand side as 
 $$\frac{1}{4\pi^2} \int_{(-\pi,\pi]^2} 
\overline{S_{\overline{g}}(\alpha_1)
S_{g}(\alpha_2)} \prod_{j=0}^d D_H({n_1}^j\alpha_1+{n_2}^j\alpha_2)\mathrm d\boldsymbol\alpha
$$ as   in \cite[Lemma 3.4]{skoro}.
Here,  for 
any $F:\mathbb N\to \mathbb C$ and
$t\in \mathbb R$ we denote 
$$ D_H(t):= \sum_{\substack{c\in \mathbb Z \\ 
|c|\leq H}} \mathrm e^{i c t}
\ \ \ \textrm{ and } \ \ \ 
S_F(t):=\sum_{c\leq (d+1) \c M^d  H}
F(c) \mathrm e^{i c t}
$$ where $\c M:= \max\{n_1,n_2\}$.

For the non-diagonal terms $n_1\neq n_2$ 
we apply 
\cite[Lemma 3.5]{skoro}  
to obtain 
$$ \sum_{\substack{ P\in \Z[t],
\deg(P)=d\\
|P|\leq H  }}
\overline{g(P(n_1))} g(P(n_2))
\ll \| S_{\overline{g}} \|_\infty 
S_{|g|}(0) H^{d-1} \frac{\c M(\log H)^2}{|n_2-n_1|}
$$ with an implied constant 
independent of $n_i$ and $H$.
By \eqref{prop:1} we get 
$\| S_{\overline{g}} \|_\infty 
\ll (\c M^d H)^{\gamma}$. 
Furthermore, Cauchy's inequality and  \eqref{prop:2}
shows that $S_{|g|}(0)\ll H \c M^d$. Thus, the contribution of the non-diagonal terms towards 
Theorem \ref{gl3} is 
$$\ll 
H^{d+\gamma} (\log H)^2
\sum_{ n_1\neq n_2\leq x}
\frac{\c M^{1+d(1+\gamma)}}{|n_2-n_1|}
\ll 
H^{d+\gamma} (\log H)^2
x^{d(1+\gamma)}
\sum_{n_2\leq x} n_2
\sum_{n_1\in [1,n_2)  }
\frac{1}{n_2-n_1}
.$$ The sum over $n_1$ in the 
outmost right-hand side is 
$O(\log n_2)$, hence, 
the overall bound is 
$$\ll H^{d+\gamma} (\log H)^2
x^{d(1+\gamma)+2} (\log x) .$$
This is $O(H^{d+1} x/(\log x)^N)$ 
if $x\leq H^{\eta-\epsilon}$.  

We have  proved     
$$\sum_{\substack{ P\in \Z[t] \\\deg(P)=d , |P|\leq H  }}
\left| \sum_{\substack{ n \leq x\\P(n)>0} }  g(P(n))\right|^2
=O\left( \frac{H^{d+1}x}{(\log x)^N}\right)+
\sum_{ n_1\leq x}
\sum_{\substack{ P\in \Z[t], P(n_1)>0
\\ \deg(P)=d, |P|\leq H}}
|g(P(n_1))|^2.$$ 
Write $P(n_1)=\sum_{j=0}^d c_j {n_1}^j$ with $|c_j|\leq H$
and $N=\sum_{j=1}^d c_j {n_1}^j$. Letting $t=c_0+N$
we use \eqref{prop:2} to write 
the sum over $P$ in the right-hand side as   
$$
\sum_{\substack{ (c_1,\ldots, c_d)\in \Z^d\\ |c_j|\leq H \forall  j }} 
\sum_{ \substack{ t>0\\t\in [N-H,N+H] }}
|g(t)|^2=
\sum_{\substack{ \b c \in \Z^d\\ |\b c|_\infty\leq H  }}
\left( 2CH+O((|N|+H)^{\beta})\right)
=(2H)^{d+1} C +O(H^d (H x^d)^{\delta}).$$
The error term   is
$O(H^{d+1}/(\log x)^N)$ if 
$ x \leq  H^{\eta-\epsilon}$.  
This concludes   the proof
of Theorem \ref{gl3}.
\section{Second moment without minor arcs}
\label{s:senzaarcigrandi}
In this section we weaken the 
assumptions of Theorem \ref{gl3}
and   prove a weaker conclusion.
Assume that there are  constants
$\sigma>0$ and $\tau\in [0,1)$ 
such that the arithmetic function 
$g: \mathbb Z\to \mathbb C$ 
satisfies the following for all $q\in \mathbb N$ and $x\geq 1$:
\beq{prop:1'}{\sup_{\substack{   a \in \Z/q\Z}}
\left|\sum_{ n\leq x}
 g(n)\mathrm e^{2\pi i \frac{a}{q}n}\right|\ll q^\sigma x^{\tau},} 
 with an implied that is independent of $q$ and $x$.
 Let \begin{equation}\label{A sfogar lo sdegno mio}
 \rho:=\min\left\{
 \frac{1-\tau}{d(1+\tau)+\max\{\sigma,1/2\}+1/2},
\frac{1-\delta}{d \delta}
\right\}.
 \end{equation}
\begin{theorem}\label{gln} Assume 
$g:\mathbb  Z\to \mathbb C$ satisfies 
 \eqref{prop:2}
and 
\eqref{prop:1'}. Fix any $\epsilon>0$ and 
 $d\in \mathbb N$. Then for all 
$ x \leq H^{\rho-\epsilon}$
we have 
$$\frac{1}{\#\{P\in \Z[t]: \deg(P)=d,|P|\leq H\}}
\sum_{\substack{ P\in \Z[t] \\\deg(P)=d , |P|\leq H  }}
\left| \sum_{\substack{  n \leq x\\ P(n)>0} } 
g(P(n))\right|^2 \ll x,$$where the implied constant
depends at most on $C,d,\epsilon,\sigma,\tau$ 
and the implied constants in  \eqref{prop:2}
and 
\eqref{prop:1'}.\end{theorem} 
While Theorem 2.2 in \cite{joni} provides a 
non-trivial bound, achieving the square-root cancellation bound
$O(x)$ seems to require some modifications in its proof.
The proof of Theorem \ref{gln} is based on
the use of Gauss kernels 
 imposed on the random coefficients of $P$,
 as   introduced in
 \cite[\S 2]{frei}. While the weights 
 in \S \ref{s:proof1} enable the circle method to 
 function under \eqref{prop:1'} rather 
 than \eqref{prop:1}, they yield only upper 
 bounds instead of the asymptotic in Theorem \ref{gl3}.

The Jacobi theta function is defined 
for $z,w\in \mathbb C$ with $\Im(w)>0$
by 
$$\theta(z;w):=\sum_{n\in \mathbb Z} 
\exp(\pi i n^2 w + 2\pi i n z ).$$ For any $H\geq 1 $ and $\alpha \in \mathbb R$ 
we let  $ K_H(\alpha) := \theta(\alpha/2\pi;i/H^2)$.
Using Poisson summation one can verify various properties of $K_H$, for example,  that 
$K_H(\alpha)$ is a  non-negative real number, see
the proof of  \cite[Lemma 2.15]{frei}. 
A straightforward calculation shows that 
for any integer $c$
the Fourier transform 
is given by $ \widehat{K}_H(c) = 
\mathrm e^{-\pi c^2/H^2}$. Now writing $P(t)=\sum_{j=0}^d c_j t^j$ we 
upper bound $\mathds 1_{[-H,H]}(c_j)$ by 
$O( \widehat{K}_H(c_j))$ for each $j$. This results in 
$$
\sum_{\substack{ P\in \Z[t] \\\deg(P)=d , |P|\leq H  }}
\left| 
\sum_{\substack{ n \leq x\\P(n)>0} } g(P(n))\right|^2
\ll
\sum_{\b c \in \Z^{d+1} } \left(\prod_{j=0}^d  \widehat{K}_H(c_j)\right)
\left| 
\sum_{\substack{ n \leq x\\P_\b c(n)>0} } 
g(P_{\b c}(n))\right|^2
$$ where $P_{\b{c} }(t)=\sum_{j=0}^d c_j t^j$.
Opening up the square we obtain 
$$\sum_{ n_1, n_2 \leq x} 
\sum_{\substack{\b c \in \Z^{d+1}\\P_\b c(n_1),
P_\b c(n_2)>0} } 
\left(\prod_{j=0}^d  \widehat{K}_H(c_j)\right)
g(P_{\b c}(n_1)) \overline{g(P_{\b c}(n_2))}
.$$ As in \cite[Equation (2.20)]{frei} we may 
now see that  \begin{equation}\label{ch'io muti consiglio}
\sum_{\substack{ P\in \Z[t] \\\deg(P)=d , |P|\leq H  }}
\left| 
\sum_{\substack{  n \leq x\\P(n)>0} } g(P(n))\right|^2
\ll \sum_{n_1,n_2\leq x}
\int_{(-\pi,\pi]^2} 
\overline{S(\alpha)} S(\beta)
\left(\prod_{k=0}^d K_H(n_1^k\alpha-n_2^k\beta) \right)
\mathrm d \alpha \mathrm d \beta
, \end{equation} where for $\gamma \in \mathbb R$ we let 
$$  S(\gamma)= \sum_{\substack{  t 
\leq (d+1) x^d H }} g(t) \mathrm e ^{i \gamma t }.$$
Note that 
$$ |S(\gamma)|\ll \sum_{  t  \leq (d+1) x^d H } |g(t)| \ll x^d H $$ by Cauchy's inequality and  \eqref{prop:2}.

Let $\|\gamma\|$ denote the distance of $\gamma$ from the closest integer multiple of $2\pi$. 
\begin{lemma}[Minor arcs for non-diagonal terms]
\label{Li prodigi della Divina Grazia nella conversione e morte di San Guglielmo}
The contribution of $n_1\neq n_2$ and $\alpha,\beta$ satisfying 
\begin{equation}
\label{Abbassa l'orgoglio}
\|\alpha-\beta\|>\frac{\log x}{H} \ \ \ \mathrm{ or } \ \ \ 
\|n_1\alpha-n_2\beta\|>\frac{\log x}{H} \end{equation}
towards the right-hand side of 
\eqref{ch'io muti consiglio}
is $\ll H^{d+1}/x$, where the implied constant is independent of 
$n_1,n_2,x$ and $H$.
\end{lemma}\begin{proof} 
We apply  the bound 
$|K_H(\alpha)| \ll |K_H(0)|\ll H$  for $k=2,3\ldots, d$ to get 
$$ \ll x^2
(x^d H)^2 H^{d-1}
\int_{\eqref{Abbassa l'orgoglio}} 
 K_H( \alpha- \beta)   K_H(n_1 \alpha-n_2 \beta)  
\mathrm d \alpha \mathrm d \beta
.$$ By  \cite[Lemma 2.9]{frei} this  is 
$$
\ll  x^2 (x^d H)^2 H^{d-1}\int_{ 
\max\{\|\alpha\|,
\|\beta\|\}>(\log x)/H }
 K_H(\alpha)  K_H(\beta) 
\mathrm d \alpha \mathrm d \beta
.$$ Denote $\delta=(\log x)/H $.
Using the bound 
$\int_{-\pi}^\pi K_H(\beta)\mathrm d \beta \ll 1$ that is verified 
at the end of the proof of \cite[Lemma 2.15]{frei}
we obtain the bound $$
\ll  x^2 (x^d H)^2 H^{d-1}\int_{ 
\|\alpha\|>\delta} 
 K_H(\alpha) 
\mathrm d \alpha \ll \frac{x^{2d+2}H^{d+1}}{\delta H 
\mathrm e^{(\delta H)^2/(4\pi))}}
,$$ where the last bound is proved via Poisson summation 
in \cite[Lemma 2.15]{frei}.
With our choice for 
$\delta=(\log x)/H$ we have 
$x^{2d+4}\leq x^{-1}\delta H 
\mathrm e^{(\delta H)^2/(4\pi))}$. 
This means that the overall contribution   is at most 
$\ll H^{d+1}/x$, which is   admissible. \end{proof}
Recall the definition of $ \rho$
in \eqref{A sfogar lo sdegno mio}.
\begin{lemma}[Major arcs for non-diagonal terms]
\label{La serva padrona}
For all $\epsilon>0$ and $x\leq H^{\zeta-\epsilon}$,
the contribution of $n_1\neq n_2$ and $\alpha,\beta$ with 
\begin{equation}
\label{Così dunque si teme}
\|\alpha-\beta\|\leq\frac{\log x}{H} \ \ \ \mathrm{ and } \ \ \ 
\|n_1\alpha-n_2\beta\|\leq \frac{\log x}{H} \end{equation}
towards the right-hand side of 
\eqref{ch'io muti consiglio}
is $\ll H^{d+1}x$, where the implied constant is independent of 
$n_1,n_2,x$ and $H$.
\end{lemma}\begin{proof} Let $q:=n_2-n_1$.
By \cite[Lemma 2.11]{frei} the condition 
\eqref{Così dunque si teme} implies that 
$$\left|\alpha - \frac{2\pi a}{q} \right|\ll 
\frac{(\log x)}{H}
\frac{|n_2|}{|q|} 
\ \ \ \textrm{ and } \ \ \ 
\left|\beta -\frac{2\pi b}{q} \right|\ll 
\frac{(\log x)}{H} \frac{|n_1|}{|q|}$$ 
for some integers 
$a,b$. 
Letting $$ A(T)=\sum_{\substack{  t \leq T }} g(t) \mathrm e^{2\pi i at/q}$$ we then 
use partial summation to see that for $\eta:=\alpha-2\pi a/q$ 
one has 
$$S(\alpha)\ll |A((d+1) x^d H) |
+ |\eta| \int_{1}^{(d+1) x^d H} |A(t)| \mathrm d t.$$ 
By \eqref{prop:1'} we get  
the overall bound   
$S(\alpha)
\ll 
|q|^\sigma
(x^d H)^{\tau} +|q|^\sigma|\eta| (x^d H)^{1+\tau}
$, which is $$
\ll
|n_2-n_1|^\sigma (x^d H)^{\tau} 
+
|n_2-n_1|^{\sigma-1} (\log x) x^{1+d(1+\tau)} H^{\tau}
=:\mathfrak F_1+\mathfrak F_2,
$$ say. Thus, $|S(\alpha) S(\beta)|\ll \mathfrak F_1^2+\mathfrak F_2^2$,
hence,  the contribution towards
the right-hand side of 
\eqref{ch'io muti consiglio} is 
$$\ll \sum_{n_1\neq n_2\leq x}
(\mathfrak F_1^2+\mathfrak F_2^2) 
\int_{(-\pi,\pi]^2} 
 \prod_{k=0}^d K_H(n_1^k\alpha-n_2^k\beta)  
\mathrm d \alpha \mathrm d \beta.$$
We use the trivial    bound 
$K_H(\gamma)\ll K_H(0) \ll H$ for all $k\notin \{0,1\}$
and then apply \cite[Lemma 2.9]{frei} to 
get the overall bound  
$$\ll   \sum_{n_1\neq n_2\leq x}
(\mathfrak F_1^2+\mathfrak F_2^2) H^{d-1}\ll
x^{2d\tau}  H^{d-1+2\tau}   \mathcal A
+
(\log x)^2 x^{2+2d(1+\tau)} H^{d-1+2\tau}  \mathcal B
,$$ where $$ \mathcal A=
\sum_{n_1\neq n_2\leq x}
|n_2-n_1|^{2\sigma} \ll x^{2+2\sigma}  
\ \ \ \textrm{ and } \ \ \ 
\mathcal B=\sum_{n_1\neq n_2\leq x}
|n_2-n_1|^{2\sigma-2}\ll  x^{\max\{1,2\sigma\}} \log x  .$$ If 
$x\leq H^{\zeta-\epsilon}$    
then one can verify   \[
x^{2d\tau}  H^{d-1+2\tau}   \mathcal A \ll H^{d+1} x \ \ \ 
\textrm{ and } \ \ \ 
 (\log x)^2 x^{2+2d(1+\tau)} H^{d-1+2\tau}  \mathcal B
\ll H^{d+1} x .\qedhere\] 
\end{proof}  
 \begin{lemma}[Diagonal terms]
 \label{A fondar le mie grandezze}
 Fix $\epsilon>0$. 
 For all $x\leq H^{\rho-\epsilon}$ we have 
$$ \sum_{ n\leq x} 
\sum_{\substack{\b c \in \Z^{d+1}\\P_\b c(n)>0 } } 
\left(\prod_{j=0}^d  \widehat{K}_H(c_j)\right)
|g(P_{\b c}(n))|^2 \ll
H^{d+1} x,$$ with an implied constant that is independent of $x$ and $H$. 
\end{lemma}\begin{proof} 
By  \eqref{prop:2} with $t=x$
we  get      
$|g(t)|^2\ll t$, hence, 
$|g(P_{\b c}(n))|^2 \ll |\b c| x^d.$
Then for any $\lambda\in \mathbb N$ we get 
$$\sum_{n\leq x} \sum_{\substack{
|\b c|>\lambda H\\ P_\b c(n)>0 }} 
\left(\prod_{j=0}^d  \widehat{K}_H(c_j)\right)
|g(P_{\b c}(n))|^2 \ll
x^{1+d } \sum_{|\b c|> \lambda H} 
|\b c|  \prod_{j=0}^d \mathrm e^{-\pi c_j^2/H^2}  
.
$$ With no loss of generality we can assume that  $|c_0|=|\b c |$, hence, 
we obtain 
$$\ll x^{1+d}  \left(\sum_{c\in \mathbb Z}\mathrm e^{-\pi c^2/H^2}   \right)^d
\sum_{  m> \lambda  H} 
m   \mathrm e^{-\pi m^2/H^2}  
\ll x^{1+d}  H^{d+1} \int_{\lambda H}^\infty  
t   \mathrm e^{-\pi t^2/H^2}  \mathrm dt 
\ll  x  H^{d+1} \frac{x^d H^2}{ \mathrm e^{\pi \lambda^2}} 
.$$   This is 
$O(H^{d+1} x)$   if   
$\lambda:=  [ 2d (\log (xH))^{1/2} ]$. 
Thus, we are left with treating  
$$\sum_{n\leq x} 
\sum_{ \substack{ |\b c|\leq 2dH (\log (xH))^{1/2}   \\P_\b c(n)>0} } 
\left(\prod_{j=0}^d  \widehat{K}_H(c_j)\right)
|g(P_{\b c}(n))|^2.$$ Fix $n,c_1,\ldots, c_d$, let $N=
\sum_{j=1}^d c_j n^j$ and 
$t=c_0+N$. Then the last sum equals $$\sum_{n\leq x} 
\sum_{\substack{
(c_1,\ldots, c_d) \in \Z^d\\  |c_j|  \leq 2d H(\log (xH))^{1/2}  }}
\left(\prod_{j=1}^d  \widehat{K}_H(c_j)\right)
\sum_{\substack{t>0\\ | t-N|\leq \
2d H(\log (xH))^{1/2} }  } 
\mathrm e^{-\pi (t-N)^2/H^2}|g(t)|^2.$$  
We upper-bound the last sum over $t$ by 
$$ 
\sum_{0\leq j \leq 2d  (\log (xH))^{1/2}  } \mathrm e^{-\pi j^2}
\sum_{\substack{ j H\leq | t-N|\leq (j+1) H}}|g(t)|^2.$$

The sum over $t$ is over over a union of at most $4$
intervals of 
 length $\gg H $
  and   centre   $N\ll x^d H (\log (xH))^{1/2}$. By \eqref{prop:2}
  we know that 
 for $1\leq y \leq x $  one has 
$$\sum_{x<n \leq x+y} |g(n)|^2 = C y+O(x^{\delta})$$
so that if $y>x^\delta$ then $\sum_{x<n \leq x+y} |g(n)|^2\ll y $.
Hence, if 
$H  \gg (x^dH (\log (xH))^{1/2})^\delta$ then the sum over $t$ is 
$$ 
\sum_{0\leq j \leq 2d  (\log (xH))^{1/2}  } \mathrm e^{-\pi j^2}
H\ll H.$$ The overall estimate then becomes 
  $$\ll \sum_{n\leq x} 
\sum_{\substack{
(c_1,\ldots, c_d) \in \Z^d\\  |c_j|  \leq H 
  \log (xH)  }}
\left(\prod_{j=1}^d  \widehat{K}_H(c_j)\right)
 H      \ll xH^{d+1}    $$
owing to  the bound 
 $\sum_{c\in \mathbb Z}\widehat{K}_H(c) \ll H$.
 The proof concludes by noting that the required 
 inequality 
 $H  \gg (x^dH (\log (xH))^{1/2})^\delta$ holds by the assumption
 $x\leq H^{\rho-\epsilon}$. 
 \end{proof}
 \subsection*{Proof of Theorem \ref{gln}} Feeding
Lemmas
\ref{Li prodigi della Divina Grazia nella conversione e morte di San Guglielmo}-\ref{La serva padrona}
into \eqref{ch'io muti consiglio} shows  that the contribution of $n_1\neq n_2$
 is admissible. The terms $n_1=n_2$ are then handled directly by
 Lemma \ref{A fondar le mie grandezze}.
\qedhere

\section{Proofs of the main results} 
\label{mentre dormi}

\begin{proof}
[Proof of Theorem \ref{thm:general_m}]
By Lemmas \ref{lem:clasconvexity}-\ref{lem:585858}
we can take any$$\delta>\frac{m^2-1}{m^2+1}, \ \ \ 
\sigma=\tau>\frac{m+1}{m+2}
$$ in  Theorem \ref{gln}. Then the constant $\rho$ 
can be taken as any positive number that is
strictly smaller than  
\[   \min \left\{ \frac{2}{2d(2m+3) + 3m + 4}, 
\frac{2}{d(m^2 - 1)} \right\},\]
which yields the value of $\alpha$ given in 
Theorem \ref{thm:general_m}.\end{proof}

\begin{proof}[Proof of Theorem \ref{thm:small_m}]
For the  $\mathrm{GL}(2)$-case 
we use Theorem \ref{gl3}
with admissible values $\gamma>1/2$ and 
$\delta>3/5$
respectively by 
Wilton's theorem \cite[Theorem 8.1]{MR1942691}
and by Lemma \ref{lem:clasconvexity}.  
Then the constant $\eta$ 
in Theorem \ref{gl3} can be taken as anything strictly smaller than $$
\min\left\{\frac{ 1}{2+3d},
\frac{2}{3d}\right \}
=\frac{ 1}{2+3d}.$$
For the  $\mathrm{GL}(3)$-case 
we invoke again Theorem \ref{gl3},
into which we feed the values 
$\gamma>3/4$ and $\delta>4/5$
respectively by the work of Miller
\cite[Theorem 1.1]{Miller} and 
Lemma \ref{lem:clasconvexity} with $m=3$.
Then the constant $\eta$ in Theorem \ref{gl3} 
can be taken as any number strictly less than  
\[\min\left\{\frac{ 1}{4+7d},
\frac{1}{4d}\right \}=\frac{ 1}{4+7d}.\qedhere \] 
\end{proof}
 
\begin{remark}\label{eightvariables}
For the special  $\mathrm{GL}(2)$-cases
coming from    Hecke--Maass cusp form $\pi$
for  $SL_2(\Z)$,
we can apply Theorem \ref{gl3}
with admissible values $\gamma>1/2$ and  $\delta>4/7$
respectively by 
Wilton's theorem \cite[Theorem 8.1]{MR1942691}
and by Lemma \ref{lem:dasgu}.  
Then the constant $\eta$ 
in Theorem \ref{gl3} can be taken as anything strictly smaller than $$
\min\left\{\frac{ 1/2}{1+3d/2},
\frac{3}{4d}\right \}=\frac{1}{2+3d},$$
which coincides  with 
Theorem \ref{thm:small_m}.
\end{remark}

\end{document}